\documentclass[%reqno,oneside,
11pt]{amsart}

\usepackage[T1]{fontenc}
\usepackage{times,mathptm}
\usepackage{amssymb,epsfig,verbatim,xypic,amscd,graphicx}
\usepackage{xcolor}
\usepackage{comment}
\usepackage{amsmath}
\usepackage{mathtools}
\usepackage{cancel,soul}
\usepackage[normalem]{ulem}
\theoremstyle{plain}

\newtheorem{thm}{Theorem}[section]

\newtheorem{pro}[thm]{Proposition}
\newtheorem{lem}[thm]{Lemma}
\newtheorem{proposition-principale}[thm]{Proposition principale}
\newtheorem{thm-principal}{Th\'eor\`eme principal}[section]

\theoremstyle{definition}

\newtheorem{eg}[thm]{Example}
\newtheorem{rem}[thm]{Remark}

\newenvironment{thm-A}
{\noindent{\bf Theorem A.--}\it}{\medskip}

\newenvironment{thm-M}
{\noindent{\bf Main Theorem.}\it}{\\}

\newenvironment{thm-AA}
{\noindent{\bf Theorem A'.}\it}{\\}

\newenvironment{thm-B}
{\noindent{\bf Theorem B.--}\it}{\medskip}

\newenvironment{thm-C}
{\noindent{\bf Theorem C.--}\it}{\\ }

\newenvironment{thm-BP}
{\noindent{\bf Bell-Poonen Theorem.--}\it}{\\ }

\newenvironment{thm-BB}
{\noindent{\bf Theorem B'.}\it}

\def\C{\mathbf{C}}
\def\bfk{\mathbf{k}}

\def\R{\mathbf{R}}

\def\Z{\mathbf{Z}}
\def\N{\mathbf{N}}

\def\id{{\mathsf{Id}}}
\def\bfx{{\mathbf{x}}}
\def\bfy{{\mathbf{y}}}

\def\sO{{\mathcal{O}}}
\def\sU{{\mathcal{U}}}
\def\sV{{\mathcal{V}}}
\def\sW{{\mathcal{W}}}
\def\sB{{\mathcal{B}}}
\def\sY{{\mathcal{Y}}}
\def\sZ{{\mathcal{Z}}}
\def\sT{{\mathcal{T}}}

\def\G{{\mathbb{G}}}

\def\bbP{\mathbb{P}}
\def\bbA{{\mathbb{A}}}
\def\bbG{{\mathbb{G}}}

\def\hotimes{\hat{\otimes}}
\def\Aut{{\sf{Aut}}}
\def\Bir{{\sf{Bir}}}
\def\char{{\sf{char}}}
\def\Mor{{\sf{Mor}}}
\def\pr{{\sf{pr}}}
\def\id{{\sf{id}}}

\def\GL{{\sf{GL}}\,}

\def\Tr{{\sf{Tr}}}
\def\T{{\sf{D}}}

\def\SU{{\sf{SU}}\,}

\def\End{{\sf{End}}\,}

\def\SU{{\sf{SU}}\,}

\def\Mor{{\text{Mor}}}

\def\supp{{\text{Support}}\,}
\def\Spec{{\rm Spec\,}}

\newcommand{\smg}[1]{\marginpar{\tiny\textcolor{olive}{#1}}}

\addtocounter{section}{0}             % Start with section 1
\numberwithin{equation}{section}       % Number formulas within sections
\begin{document}
	
	\setlength{\baselineskip}{0.54cm}        % Previous 0.56
	%
	%%%%%%%%%%%%%%%%%%%%%%%%%%%%%%%%%%%%%%%%%%%%%%%%%%%%%%%%%%%%%%%%%%
	%
	\title[Characterization of the affine space]
	{Families of commuting automorphisms, and a characterization of the affine space}
%	\date{2019}
	\author{Serge Cantat, Andriy Regeta, and Junyi Xie}
%	\author[F. Author]{Serge Cantat}
%	\author[S. Author]{Andriy Regeta}
%	\author[T. Author]{Junyi Xie}
%\address{Univ Rennes, CNRS, IRMAR - UMR 6625, F-35000 Rennes, France}
%\email[F. Author]{serge.cantat@univ-rennes1.fr}
%\email[S. Author]{sauthor@theschool.edu}
	
	\address{Univ Rennes, CNRS, IRMAR - UMR 6625, F-35000 Rennes, France}
	\email{serge.cantat@univ-rennes1.fr, xiejunyi@bicmr.pku.edu.cn}
	
	\address{\noindent Institut f\"{u}r Mathematik, Friedrich-Schiller-Universit\"{a}t Jena, \newline
\indent  Jena 07737, Germany}
\email{andriyregeta@gmail.com}
	\address{BICMR, Peking University, Haidian District, Beijing 100871, China}

\email{xiejunyi@bicmr.pku.edu.cn}

	%
	%%%%%%%%%%%%%%%%%%%%%%%%%%%%%%%%%%%%%%%%%%%%%%%%%%%%%%%%%%%%%%%%%%
	%
	
	%
	%%%%%%%%%%%%%%%%%%%%%%%%%%%%%%%%%%%%%%%%%%%%%%%%%%%%%%%%%%%%%%%%%%
	%
	
	%
	%%%%%%%%%%%%%%%%%%%%%%%%%%%%%%%%%%%%%%%%%%%%%%%%%%%%%%%%%%%%%%%%%%
	%

%	\begin{abstract} 
%		We prove that an affine space, defined over  an uncountable algebraically closed field, is determined by its automorphism group in the category of affine connected varieties.
		
%	\end{abstract}
	
	\maketitle
	
	\begin{abstract}
We prove that the affine space of dimension $n\geq 1$ over an uncountable algebraically closed field $\bfk$ is determined, 
among connected affine varieties, by its automorphism group (viewed as an abstract group). 
The proof is based on a new result concerning algebraic families of pairwise commuting automorphisms.
	\end{abstract}
%
%%%%%%%%%%%%%%%%%%%%%%%%%%%%%%%%%%%%%%%%%%%%%%%%%%%%%%%%%%%%%%%%%%
%
\section{Introduction}\label{par:Intro}
%
%%%%%%%%%%%%%%%%%%%%%%%%%%%%%%%%%%%%%%%%%%%%%%%%%%%%%%%%%%%%%%%%%%
%
 
\subsection{Characterization of the affine space} In this paper, $\bfk$ is an algebraically closed field and  $\bbA^n_\bfk$  denotes the affine space of dimension $n$ over $\bfk$. 
 
\medskip
		
\begin{thm-A} 
Let $\bfk$ be an algebraically closed and uncountable field. Let $n$ be a positive integer.
Let $X$ be a reduced, connected, affine variety over $\bfk$. If  its automorphism group $\Aut(X)$  is isomorphic to $\Aut(\bbA^n_\bfk)$ as an abstract group,
then $X$ is isomorphic to $\bbA^n_\bfk$ as a variety over $\bfk$. 
\end{thm-A}

Note that no assumption is made on $\dim(X)$; in particular, we do not assume $\dim(X)=n$.
This theorem is our main goal. 
It would be great to lighten the hypotheses on $\bfk$, 
but besides that the following remarks show the result is optimal:

\smallskip
	
$\bullet$ The affine space $\bbA^n_\bfk$ is not determined by its automorphism group in the category of quasi-projective varieties because
\begin{enumerate}
\item  $\Aut(\bbA^n_\bfk)$ is naturally isomorphic to $\Aut(\bbA^n_\bfk \times Z)$ for any   projective variety $Z$  with $\Aut(Z)=\{\id\}$;
\item  for every 
algebraically closed field $\bfk$ there is a projective variety $Z$ over $\bfk$ such that $\dim(Z)\geq 1$ and $\Aut(Z)=\{\id\}$ (one can take
a general curve of genus $\ge 3$; see \cite{Poonen2000} and \cite[Main Theorem]{Popp}). 
\end{enumerate}

\smallskip

$\bullet$ The connectedness is crucial:  
$\Aut(\bbA^n_\bfk)$ is isomorphic to the automorphism group of 
the disjoint union of $\bbA^n_\bfk$ and $Z$ if $Z$ is a variety with $\Aut(Z)=\{\id\}$. 

\subsection{Previous results} The literature contains already several theorems that may be compared to Theorem~A. We refer to \cite{Deserti:Algebra} for an interesting introduction 
and for the case of the complex affine plane; see \cite{Kraft2017, KRS18} for extensions 
and generalisations of D\'eserti's results in higher dimension.  Some of those results assume
 $\Aut(X)$ to be isomorphic to $\Aut(\bbA^n_\bfk)$ as an ind-group; this is a rather strong hypothesis. Indeed, there are 
examples of affine varieties $X$ and $Y$ such that $\Aut(X)$ and 
$\Aut(Y)$ are isomorphic as abstract groups, but not isomorphic as ind-groups (see \cite[Theorem 2]{1710.06045}).
In \cite{LRU18} the authors prove that  an affine toric surface is determined by its group of automorphisms in the category of affine surfaces; unfortunately, their methods  do not work in higher dimension. 

\subsection{Commutative families}\label{par:1.3} The proof of Theorem~A relies on a new result concerning families of pairwise commuting automorphisms of affine varieties.
To state it, we need a few standard notions.
If $V$ is a subset of a group $G$, we denote by $\langle V \rangle$ the subgroup generated by $V$, 
i.e. the smallest subgroup of $G$ containing $V$. We say that $V$ is  {\bf{commutative}} if $fg=gf$ for all pairs
or equivalently, if $\langle V\rangle$ is an abelian group. 
In the following statement, $\Aut(X)$ is viewed as an ind-group, so that it makes sense to speak of algebraic subsets of it (see the
definitions in Section~\ref{ind-groups}). 

\medskip
	
\begin{thm-B} 
Let $\bfk$ be an algebraically closed field and let $X$ be an affine variety over $\bfk$.
Let $V$ be a commutative   irreducible algebraic subvariety of  $\Aut(X)$ containing the identity. 
Then $\langle V \rangle$ is an algebraic subgroup of $\Aut(X)$.
\end{thm-B}

It is crucial to assume that $V$ contains the identity. Otherwise,  a counter-example would be given by
a single automorphism $f$ of $X$ for which the sequence $n\mapsto \deg(f^n)$ is not bounded (see Section~\ref{par:degrees}). 
To get a family of positive dimension, consider the set $V$ of automorphisms $f_a\colon (x,y)\mapsto (x,axy)$ of $(\bbA^1_\bfk\setminus\{0\})^2$, 
for $a\in \bfk\setminus \{ 0\}$; $V$ is commutative and irreducible, but $\langle V\rangle$ has infinitely many connected components (hence $\langle V\rangle$
is not algebraic). 
However,  if $V$ satisfies the hypotheses of Theorem~B except that it does not contain the identity, the subset $V\cdot V^{-1}\subseteq \Aut(X)$ is irreducible, commutative and contains the identity; if its dimension is positive, Theorem~B implies that $\Aut(X)$ contains a commutative algebraic subgroup of positive dimension.

\begin{rem}\label{rem:intro-nested}
As noted by the referee and H. Kraft, 
Theorem B is equivalent to the following statement. 
{\sl{Let $X$ be an affine variety, over an algebraically closed field $\bfk$.  
Then, any connected commutative ind-subgroup $G$ of $\Aut(X)$ is a union
of commutative algebraic subgroups.}} With the vocabulary of~\cite[\S 0.9 and 9.4]{FK18}, this means that {\sl{the connected commutative ind-subgroups of $\Aut(X)$ are nested}}.  Indeed, $G$ is an increasing union of irreducible, connected, and
commutative subvarieties $V_i$, $i\geq 1$, containing $\id_X$ (see~\S~\ref{par:ind_groups_connected} below). By Theorem~B, $G$ is  the increasing union of the algebraic subgroups $G_i:=\langle V_i\rangle$. 
\end{rem}

\subsection{Acknowledgement}	 We thank Jean-Philippe Furter,  Hanspeter Kraft, and 
Christian Urech for interesting discussions and comments, and the 
referee for helpful criticisms and suggestions. Hanspeter Kraft provided many interesting remarks that greatly 
improved this article, in particular the proof of Theorem~\ref{corconnuncountablecommutator}, using the ind-group structure of centralizers, 
is simpler than the first proof we had written. 
%
%%%%%%%%%%%%%%%%%%%%%%%%%%%%%%%%%%%%%%%%%%%%%%%%%%%%%%%%%%%%%%%%%%
%
\section{Degrees and ind-groups}
%
%%%%%%%%%%%%%%%%%%%%%%%%%%%%%%%%%%%%%%%%%%%%%%%%%%%%%%%%%%%%%%%%%%
%

\subsection{Degrees and compactifications} \label{par:degrees}

Let $X$ be an affine variety. Embed $X$ in the affine space $\bbA^N_\bfk$ for some $N$, and denote
by $\bfx=(x_1, \ldots, x_N)$ the affine coordinates of $\bbA^N_\bfk$. Let $f$ be an 
automorphism of $X$. Then, there are $N$ polynomial functions $f_i\in \bfk[\bfx]$ such that 
$f(\bfx)=(f_1(\bfx), \ldots, f_N(\bfx))$ for $\bfx\in X$. One says that $f$ has degree $\leq d$ if one can choose
the $f_i$ of degree $\leq d$; the degree $\deg(f)$ can then be defined as the minimum of these degrees $d$. 
This notion depends on the embedding $X \hookrightarrow \bbA^N_\bfk$.

Another way to proceed is as follows. 
To simplify the exposition, assume that all irreducible components of $X$ have the same dimension $k=\dim(X)$. 
Fix a compactification ${\overline{X}}_0$ of $X$ by a projective variety and let ${\overline{X}}\to {\overline{X}}_0$
be the normalization of ${\overline{X}}_0$. If $H$ is an ample line bundle on ${\overline{X}}$, and if $f$ is a birational transformation of ${\overline{X}}$, 
one defines $\deg_H(f)$ (or simply $\deg(f)$) to be the intersection number 
\begin{equation}
\deg(f)=(f^*H)\cdot (H)^{k-1}.
\end{equation}
Since $\Aut(X)\subset \Bir({\overline{X}})$, we obtain a second notion of degree.  
It is shown in~\cite{Dang-degrees-2020, Truong} (see also \S~\ref{par:appendix} below)
that these notions of degrees are compatible: if we change the embedding $X \hookrightarrow \bbA^N_\bfk$,
or the polarization $H$ of  ${\overline{X}}$, or the compactification  ${\overline{X}}$, 
we get different degrees, but any two of these degree
functions are always comparable, in the sense that there are positive constants satisfying
\begin{equation}
a\deg(f)\leq \deg'(f)\leq b\deg(f) \quad \; (\forall f \in \Aut(X)).
\end{equation}

A subset $V\subset \Aut(X)$ is of {\bf{bounded degree}} if there is a uniform upper bound $\deg(g) \leq D<+\infty$ for all $g\in V$. 
This notion does not depend on the choice of degree. If $V\subset \Aut(X)$ is of bounded degree, 
then $V^{-1}=\{ f^{-1}\; ; \; f\in V\} \subset \Aut(X)$ is of bounded degree too (see \cite{Dang-degrees-2020} and 
\cite{FK18} for instance); we shall not use this result.

%%%
\subsection{Automorphisms of affine varieties and  ind-groups}\label{ind-groups}
%%%
The notion of an ind-group goes back to Shafarevich, who called  these objects infinite dimensional groups in \cite{Shafarevich1966}. We refer to  \cite{FK18, KPZ15}  for detailed introductions to this notion. 

\subsubsection{Ind-varieties}  
By an {\bf{ind-variety}}  we mean a set $\mathcal{V}$ together with an ascending filtration $\mathcal{V}_0 \subset \mathcal{V}_1 \subset \mathcal{V}_2 \subset ... \subset \mathcal{V}$ such that the following is satisfied:
		
		(1) $\mathcal{V} = \bigcup_{k \in \N} \mathcal{V}_k$;
				
		(2) each $\mathcal{V}_k$ has the structure of an algebraic variety over $\bfk$;
		
		(3) for all $k \in \N$ the inclusion $\mathcal{V}_k \subset \mathcal{V}_{k+1}$ is a closed immersion.

\noindent{}We refer to \cite{FK18} for the notion of equivalent filtrations on ind-varieties. 

A map $\Phi: \mathcal{V} \to \mathcal{W}$ between ind-varieties $\mathcal{V} = \bigcup_k \mathcal{V}_k$ and $\mathcal{W} = \bigcup_l \mathcal{W}_l$ is a {\bf{morphism}}  if for each $k \in \N$ there is $l \in \N$ such that 
$\Phi(\mathcal{V}_k) \subset \mathcal{W}_l$ and the induced map $\Phi \colon {\mathcal{V}_k} \to \mathcal{V}_l$ is a morphism of algebraic varieties. Isomorphisms of ind-varieties are defined in the usual way. %MM
An ind-variety $\mathcal{V} = \bigcup_k \mathcal{V}_k$ has a natural Zariski topology: $S \subset \mathcal{V}$ is {\bf{closed}} (resp. {\bf{open}}) if $S_k := S \cap  \mathcal{V}_k \subset \mathcal{V}_k$ is closed (resp. open) for every $k$. A closed subset $S \subset \mathcal{V}$ inherits a natural structure of ind-variety and is called an {\bf{ind-subvariety}}. An ind-variety $\mathcal{V}$ is said to be affine if each $ \mathcal{V}_k$ is affine. We shall only consider affine ind-varieties and  for simplicity we just call them ind-varieties. An ind-subvariety $S$ is an {\bf{algebraic subvariety}} of $\mathcal{V}$ if $S \subset \mathcal{V}_k$ for some $k \in \N$; by definition, a {\bf{constructible subset}} will always be a constructible subset in an algebraic subvariety of $\mathcal{V}$.
 
\subsubsection{Ind-groups}\label{par:ind_groups_connected}
The  product of two ind-varieties is defined in the obvious way. An ind-variety $\mathcal{G}$ is called an {\bf{ind-group}} if the underlying set $\mathcal{G}$ is a group and
the map $\mathcal{G} \times \mathcal{G} \to \mathcal{G}$, defined by $(g,h) \mapsto gh^{-1}$, is a morphism of ind-varieties. If a subgroup $H$ of $\mathcal{G}$ is closed for the Zariski topology, then $H$ is naturally an ind-subgroup of $\mathcal{G}$; it is an {\bf{algebraic subgroup}} if it is an algebraic subvariety of $\mathcal{G}$.
A connected component of an ind-group $\mathcal{G}$, with a given filtration $\mathcal{G}_0 \subset \mathcal{G}_1 \subset \mathcal{G}_2 \subset \dots$,
is an increasing union of connected components $\mathcal{G}_i^c$ of $\mathcal{G}_i$.
The {\bf{neutral component}} $\mathcal{G}^\circ$ of $\mathcal{G}$ is  the union of the connected components of the $\mathcal{G}_i$ 
containing  the neutral element $\id \in \mathcal{G}$.
We refer to  \cite{FK18}, and in particular to Propositions 1.7.1 and 2.2.1, showing that
{\sl{$\mathcal{G}^\circ$ is an ind-subgroup in $\mathcal{G}$ whose index is at most countable}} (the proof of \cite{FK18} works in arbitrary characteristic).
	
We say that the ind-group $\mathcal{G}$ acts \textbf{morphically} on $X$ if the action $\mathcal{G} \times X \to X$ of  $\mathcal{G}$ on $X$
induces  a morphism of algebraic varieties $\mathcal{G}_i \times X \to X$ for every $i \in \N$.

\begin{thm}\label{thm:ind-group}
Let $X$ be an affine variety over an algebraically closed field~$\bfk$. Then $\Aut(X)$ has the structure of an ind-group
acting morphically on $X$.\end{thm}

In particular,  if $V$ is an algebraic subset of $\Aut(X)$,
then $V(x) = \{ v(x) \mid v \in V \}\subset X$ is constructible for every $x\in X$ by Chevalley's theorem.  
The proof can be found in \cite[Proposition 2.1]{KPZ15} (see also~\cite{FK18}, Theorems~5.1.1 and~5.2.1): the authors assume that the field has characteristic $0$, but their proof works in the general setting. To obtain a filtration, one starts with a closed embedding $X \hookrightarrow \bbA^N_\bfk$, and
define $\Aut(X)_{d}$ to be the set of automorphisms $f$ such that $\max\{\deg (f),\deg(f^{-1})\}\leq d$. For example, if $X = \bbA^n_\bfk$, the ind-group filtration 
$(\Aut(\bbA^n_\bfk)_d)$ of $\Aut(\bbA^n_\bfk)$ is defined by the following property: an automorphism $f$   
is in $(\Aut(\bbA^n_\bfk)_d)$ if the polynomial formulas for $f=(f_1, \ldots, f_n)$ and for its inverse $f^{-1}=(g_1, \ldots, g_n)$ satisfy
\begin{equation}
\deg f_i \leq d \;  {\text{ and }} \; \deg g_i \leq d, \; \; (\forall i\leq n).
\end{equation}
Note that an ind-subgroup is algebraic if and only if it is of bounded degree. Thus, we get the following basic fact. 
 
\begin{pro}\label{ind-group-bded-degree}
Let $X$ be an affine variety over an algebraically closed field~$\bfk$. 
Let $V$ be an irreducible algebraic subset of $\Aut(X)$ that contains $\id$. Then $\langle V\rangle$ is an algebraic subgroup of $\Aut(X)$, 
acting algebraically on $X$, if and only if $\langle V \rangle$ is of bounded degree.
\end{pro}

\begin{eg} 
Let $g\in \SU_2(\C)$ be an irrational rotation, and set $V=\{g\}\subset \Aut(\bbA^2_\C)$. Then $\langle V\rangle$ 
is not an algebraic group, but it is Zariski dense in an abelian algebraic subgroup of $\GL_2(\C)\subset \Aut(\bbA^2_\C)$.
This shows that $\id \in V$ is a necessary hypothesis. 
\end{eg}

\begin{proof}(See also Chap I, Prop. 2.2 of~\cite{Borel:Linear_Algebraic_Groups}).--
If $\langle V\rangle$ is algebraic, then it is contained in some $\Aut(X)_{d}$ and, as such, is of bounded degree; 
moreover, Theorem~\ref{thm:ind-group} implies that the action $\langle V\rangle \times X\to X$ is algebraic. %MM
If  $\langle V\rangle$ is of bounded degree, then $\langle V\rangle^{-1}=\langle V\rangle$ is of bounded degree too, and $\langle V\rangle$ is contained in some $\Aut(X)_{d}$.
The Zariski closure $\overline{\langle V\rangle}$ of $\langle V\rangle$ in $\Aut(X)_{d}$ is an algebraic subgroup of $\Aut(X)$; 
we are going to show that $\overline{\langle V\rangle}=\langle V\rangle$. 
Set $W=V\cdot V^{-1}$, and note that $W$ contains $V$ because $\id\in V$.
By definition, $\langle V\rangle$ is the increasing union 
of the subsets $W\subset W\cdot W\subset \cdots \subset W^k\subset \cdots$, 
and by Chevalley theorem, each $W^k\subset  \overline{\langle V\rangle}$ is constructible. The $\overline{W^k}$ are irreducible, because $V$ is irreducible, and their dimensions are bounded by the dimension of $\Aut(X)_{d}$;
so, there exists $\ell\geq 1$ such that $\overline{W^\ell}=\cup_{k\geq 1}\overline{W^k}\subseteq \overline{\langle V\rangle}$.
Since $\langle V\rangle\subseteq \cup_{k\geq 1}\overline{W^k}$, we get $\overline{W^\ell}=\overline{\langle V\rangle}$; thus, 
there exists a Zariski dense open subset $U$ of $\overline{\langle V\rangle}$ which is contained in $W^\ell$.
Now, pick any $f$ in $\overline{\langle V\rangle}$.
Then $(f\cdot U)$ and $U$ are two Zariski dense open subsets of $\overline{\langle V\rangle}$, so  $(f\cdot U)$ intersects $U$
and this implies that $f$ is in $U\cdot U^{-1}\subset \langle V\rangle$. So $\overline{\langle V\rangle}\subset \langle V\rangle$.
\end{proof}

%
%%%%%%%%%%%%%%%%%%%%%%%%%%%%%%%%%%%%%%%%%%%%%%%%%%%%%%%%%%%%%%%%%%
%
\section{Algebraic varieties of commuting automorphsims}
%
%%%%%%%%%%%%%%%%%%%%%%%%%%%%%%%%%%%%%%%%%%%%%%%%%%%%%%%%%%%%%%%%%%
%
	
Let $\bfk$ be an algebraically closed field.
Let $X$ be an affine variety over $\bfk$ of dimension $d$. 
In this section, we prove Theorem~B. 
Since $V\subset \Aut(X)$ is irreducible and contains the identity, 
every irreducible component of $X$ is invariant under the action of $V$ (and of $\langle V\rangle$); thus, we may and do 
assume $X$ to be irreducible. 
	 	
%%%%%%%%%%%%%%%%%%% 
\subsection{Invariant fibrations, base change, and degrees}\label{par:base_change_and_fibrations}
%%%%%%%%%%%%%%%%%%%
Let $B$ and $Y$ be affine varieties and assume that $B$ is irreducible.
Let $\pi\colon Y\to B$ be a dominant morphism.  
By definition, $\Aut_\pi(Y)$ is the group of automorphisms $g\colon Y\to Y$ such that $\pi\circ g = \pi$. 
Note that $\Aut_\pi(Y)$ is a closed ind-subgroup of $\Aut(Y)$. %\amr{ We added the sentence concerning closedness of $\Aut_\pi(Y)$.}
	
Let $B'$ be another irreducible affine variety, and let $\psi\colon B'\to B$ be a  quasi-finite and 
dominant  morphism.  
Pulling-back $\pi$ by $\psi$, we get a new affine variety $Y\times_B B'=\{(y,b')\in Y\times B'; \; \pi(y)=\psi(b')\}$;  
the projections $\pi_Y \colon Y\times_B B'\to Y$ and $\pi'\colon Y\times_B B'\to B'$  satisfy $\psi\circ \pi'=\pi\circ \pi_Y$.
There is a natural homomorphism
\begin{equation}
\iota_\psi\colon \Aut_\pi(Y)\to \Aut_{\pi'}(Y\times_B B')
\end{equation}
defined by $\iota_\psi(g)=g\times_B \id_{B'}$. For every $g\in \Aut_\pi(Y)$, we have 
\begin{equation}
g\circ\pi_Y=\pi_Y\circ \iota_{\psi}(g) \quad  {\text{ and }} \quad \pi'=\pi'\circ \iota_{\psi}(g).
\end{equation}
If $\iota_{\psi}(g)=\id$ then $g\circ\pi_Y=\pi_Y$ and  $g=\id$ because  $\pi_Y$ is dominant; hence, {\sl{$\iota_{\psi}$ is an embedding}}. 
Since $\pi_Y$ is dominant and generically finite, the next lemma follows from Proposition \ref{progenfinidegrefun}.

\begin{lem}\label{lem:basechange} 
If $S$ is a subset of $\Aut_\pi(Y)$, then $S$ is of bounded degree if and only 
if its image $\iota_{\psi}(S)$ in $\Aut_{\pi'}(Y\times_B B')$ is of bounded degree.
\end{lem}
	 	
Let us come back to the example $f(x,y)=(x,xy)$ from Section~\ref{par:1.3}. This is an automorphism of the multiplicative
group $\bbG_m\times \bbG_m$ that preserves the projection onto the first factor. The degrees of the iterates
$f^n(x,y)=(x,x^ny)$ are not bounded, but on every fiber $\{x=x_0\}$, the restriction of
$f^n$ is the linear map $y\mapsto (x_0)^n y$, of constant degree $1$. More generally, if $x\in B \mapsto A(x)$ is a regular map with values in 
$\GL_N(\bfk)$, then $g\colon (x,y)\mapsto (x,A(x)(y))$ is a regular automorphism of $B\times \bbA^N_\bfk$ and, in most cases, we observe the
same phenomenon:  the degrees of the restrictions $g^n_{\vert \{x_0\} \times \bbA^N_\bfk}$ are bounded, but  the degrees of
$g^n$ are not. 

 If $X$ is an affine variety over $\bfk$ with a morphism $\pi: X\to B$, we denote by $\eta$ the generic point of $B$ and
$X_{\eta}$ the generic fiber of $\pi$.
If $G$ is a subgroup of $\Aut_{\pi}(X)$, then its restriction to $X_{\eta}$ may have bounded degree even if 
$G$ is not a subgroup of $\Aut(X)$ of bounded degree: this is shown by the previous example.  

The next proposition provides a converse result. 
To state it, we make use of the following notation. Let $B$ be an irreducible affine variety, and let $\sO(B)$ be the $\bfk$-algebra of its regular functions. 
By definition, $\bbA_{B}^N$ denotes the affine space
$\Spec \sO(B)[x_1,\dots,x_N]$
over the ring $\sO(B)$ and $\Aut_{B}(\bbA_{B}^N)$ denotes the group of {\bf{$\sO(B)$-automorphisms}} of $\bbA_{B}^N$; $\Aut_{B}(\bbA_{B}^N)$ is just another notation for $\Aut_{\pr_B}(\bbA^N\times B)$, where $\pr_B \colon \bbA^N\times B \to B$ is the projective map to the second factor (see the first lines of \S~\ref{par:base_change_and_fibrations}).
Similarly, $\GL_{N}(\sO(B))$ is the linear group over the ring $\sO(B)$. The inclusion $\GL_{N}(\sO(B))\subset \Aut_{B}(\bbA_{B}^N)$ is an embedding of ind-groups.
Indeed, the group $\GL_N(\sO(B))$ may be identified to  space of morphisms $\Mor(B, \GL_N(\bfk))$ between the affine varieties $B$ and $\GL_N(\bfk)$. As a subgroup of $\Aut_{B}(\bbA_{B}^N)$
it is closed, because it coincides with 
\begin{align*}
\{ f\in & \Aut_B(\bbA^N_B)\; ;\;    \deg f,\deg f^{-1}\le 1 \text{ and } f \text{ fixes} \\ & \text{the zero section } 0\times B\subseteq \bbA^N\times B=A^N_B \}.
\end{align*}
%$\{ f\in \Aut_B(\bbA_{B}^N)\; ;\;   \deg(f), \deg(f^{-1}) = 1\}$.

%(\footnote{This group $\GL_n(\sO(B))$ may be identified to  space of morphisms $\Mor(B, \GL_n)$ between the affine varieties $B$ and $\GL_n$. As a subgroup of $\Aut_{B}(\bbA_{B}^N)$
%it is closed, because it coïncides with $\{ f\in \Aut(\bbA^N)\; ;\;   \deg(f), \deg(f^{-1})\leq 1\}$ which is a piece of the filtration defining $\Aut(\bbA^N)$ }). 

\begin{pro}\label{proweil} Let $X$ be an irreducible and normal affine variety over $\bfk$ with a dominant morphism $\pi: X\to B$  to an irreducible affine variety  $B$ over $\bfk$. 
Let $\eta$ be the generic point of $B$ and $X_{\eta}$ the generic fiber of $\pi$.
Let $G$ be a subgroup of $\Aut_{\pi}(X)$ whose restriction to $X_{\eta}$ is of bounded degree. 
Then there exists 
\begin{itemize}
\item[(a)] a nonempty affine open subset $B'$ of $B$,  
\item[(b)] an embedding $\tau: X_{B'}:=\pi^{-1}(B')\hookrightarrow \bbA_{B'}^r$ over $B'$ for some $r\geq 1$, 
\item[(c)] and an embedding $\rho: G\hookrightarrow \GL_{r}(\sO(B'))\subseteq \Aut_{B'}(\bbA_{B'}^r)$,
\end{itemize}
%\xmg{Changed $N$ to $r$.}
such that $\tau\circ g=\rho(g)\circ \tau$  for every $g\in G$.
\end{pro}

\noindent{\bf{Notation.--}} 
For $f\in \Aut(X)$ and   $\xi\in O(X)$ (resp. in $\bfk(X)$), 
we denote by $f^*\xi$ the function $\xi\circ f$. The field of constant functions is identified with $\bfk\subset O(X)$.

\proof[Proof of Proposition \ref{proweil}]
Shrinking $B$, we assume $B$ to be normal. 
 
Pick any closed embedding $X\hookrightarrow\bbA^\ell_B\subseteq \bbP^\ell_B$ over $B$. 
Let $X'$ be the Zariski closure of $X$ in $\bbP^\ell_B$. Let $\overline{X}$ be the normalization of $X'$, with the
structure morphism $\overline{\pi}: \overline{X}\to B$; thus, $\overline{\pi}: \overline{X}\to B$ is a normal and projective scheme over $B$ containing $X$ as a Zariski open subset.
By Proposition 3.1 in \cite[Chap. II]{Hartshorne1970}, 
$D:=\overline{X}\setminus X$ is an effective Weil divisor of $\overline{X}$.
Denote by $\overline{X}_{\eta}$ the generic fiber of $\overline{\pi}$ and by $D_{\eta}$ the generic fiber of $\overline{\pi}|_D$.
Shrinking $B$ again if necessary, we may assume that all irreducible components of $D$ meet the generic fiber, i.e. $D=\overline{D_{\eta}}$.

Write $X=\Spec A$, where $A = O(X)$.
Let $M$ be a finite dimensional subspace of $A$ such that $1\in M$ and $A$ is generated by $M$ as a $\bfk$-algebra. 
Since the action of $G$ on $X_{\eta}$ is of bounded degree, there exists $m\geq 0$ such that the divisor 
\begin{equation}
({\rm Div}(g^*v)+mD)|_{\overline{X}_{\eta}}
\end{equation}
is effective for every $v\in M$ and $g\in G$.
Now, consider ${\rm Div}(g^*v)+mD$ as a divisor of $\overline{X}$ and write ${\rm Div}(g^*v)+mD=D_1-D_2$ where $D_1$ and $D_2$ are effective and have no common irreducible component. Since $g\in \Aut_{\pi}(X)$, we get $g^*v\in A$ and $D_2\cap X=\emptyset$. Moreover, $D_2\cap \overline{X}_{\eta}=\emptyset$.
So $D_2$ is contained in $\overline{X}\setminus X$, but then we deduce that $D_2$ is empty because $\overline{X}\setminus X$ is covered by $D$ and $D=\overline{D_{\eta}}$.
 
 Observe that $H^0(\overline{X}, mD)$ is a finitely generated $O(B)$-module.
Denote  by $N$ the $G$-invariant $\sO(B)$-submodule of $A$ generated by the $g^*v$, for $g\in G$ and $v\in M$. 
Since $N\subseteq H^0(\overline{X}, mD)$, $N$ is a finitely generated $\sO(B)$-module. 
Let $r$ be the dimension of the $\bfk(B)$-vector space 
$N\otimes_{\sO(B)}\bfk(B)$. Fix a basis $(w_1, \ldots, w_r)$  of this space made of elements $w_i\in N$. 
After shrinking $B$, we may assume that $N$ is a free $\sO(B)$-module generated by $w_1,\dots, w_r$. Let $W$ be a 
free $\sO(B)$-module of rank $r$ with a basis $(z_1, \ldots, z_r)$; thus, $W=\oplus_{i=1}^r\sO(B) z_i$ and
\begin{equation}
\Spec \sO(B)[W]=\Spec \sO(B)[z_1,\dots, z_r]=\bbA^r_{\sO(B)}.
\end{equation}
Let $\tau_W^*: W\to N$ be the isomorphism of modules defined by $\tau_W^*(z_i)= w_i$. The action of $G$ 
on $N$ induces a representation $\rho\colon G\to \GL_B(W)$ such that $\tau_W^*\circ\rho(g)= g^*\circ \tau_W^*$. % 

Using the basis $(z_i)$, we obtain a homomorphism $\rho\colon G\to\GL_r(\sO(B))$. Let $\tau$ 
be the morphism $X\hookrightarrow \Spec \sO(B)[W]=\bbA^r_{\sO(B)}$ over $B$ induced by $\tau_W^*: W\to N\subseteq A$. The group $\GL_r(\sO(B))$ can naturally be identified to a subgroup of $\Aut_B(  \bbA^r_{\sO(B)} )$, 
and then $\tau\circ g=\rho(g)\circ \tau$ for every $g\in G$.
\endproof	
	
%%%%%%%%%%%%%%%%%%% 
\subsection{Orbits}\label{par:orbits}
%%%%%%%%%%%%%%%%%%%
If $S$ is a subset of $\Aut(X)$ and $x$ is a point of $X$  the $S$-{\bf{orbit}} of $x$ is the
subset $S(x)=\{f(x);\; f\in S\}$. 
Let $V$ be an irreducible algebraic subvariety of $\Aut(X)$ containing $\id$. Set $W= V\cdot V^{-1}$; it is 
a constructible subset of $\Aut(X)$ containing $V$ (for $\id \in V$).
Then,  the group $\langle V\rangle$ is the union
of the sets
\begin{equation}
W^k=\{ f_1\circ \cdots \circ f_k; \; f_j\in W \; {\text{for all}} \; j \}.
\end{equation} 
 Since $W$ contains $\id$, 
the $W^k$ form a non-decreasing sequence 
\begin{equation}\label{eq:inclusion-Wk}
W^0=\{\id\}\subset W\subset W^2\subset \cdots \subset W^k\subset \cdots
\end{equation} 
of constructible subsets of $\Aut(X)$; their  closures are irreducible, because so is $V$.
 In particular, $k\mapsto \dim(W^k)$ is non-decreasing. 

The $W^k$-orbit of a point $x\in X$  is  the image of $W^k\times \{ x\}$ by the morphism $\Aut(X)\times X\to X$ defining the action on $X$:
applying Chevalley's theorem one more time, $W^k(x)$ is a constructible subset of $X$.
If $U\subset X$ is open, its $W^k$-orbit $W^k(U)$   is open too; 
thus, $\langle W\rangle (U)=\cup_{k\geq 0}W^k(U)$ is open in $X$. 

An increasing union of irreducible constructible sets needs not be stationary:  the sequence of subsets 
of $\bbA^2_\C$ defined by $Z_k=\left(\bbA^2_\C\setminus \{y=0\}\right)\cup_{j=1}^k \{(j,0)\}$ provides such an example. 
However, we shall see in the next proposition that the $W^k(x)$ are better behaved. 

Let $\pi_1$ and $\pi_2$ be the projections from $X\times X$ to the first and second factor, respectively. 
Let $\Delta_X$ be the diagonal in $X\times X$; if $Y$ is a subvariety of $X$, set
\begin{equation}
\Delta_Y=\pi_1^{-1}(Y)\cap \Delta_X=\{(y,y)\in X\times X; \; y\in Y\}\subset X\times X.
\end{equation} 
Consider the morphism $\Phi\colon \Aut(X)\times X\to X\times X$ defined by 
\begin{equation}
\Phi(g,x)=(x,g(x)),
\end{equation}
and set $\Gamma_i=\Phi(W^i\times X)$ for $i\in \Z_{>0}$. The family $(\Gamma_i)_{i \in \N}$ forms a 
non-decreasing sequence of constructible sets; we denote by $\Gamma_\infty$ their union. Then, consider the action
of $\Aut(X)$ on $X\times X$ given by 
$g\cdot(x,y)=(x,g(y))$. By construction, $\Gamma_i=W^i\cdot\Delta_X$ and
$\Gamma_\infty =\langle W\rangle\cdot\Delta_X$; similarly  $W^i\cdot \Delta_Y=\Gamma_i\cap \pi_1^{-1}(Y)$  and $\langle W\rangle\cdot \Delta_Y=\Gamma_{\infty}\cap \pi_1^{-1}(Y)$ for every subvariety $Y\subset X$.

\begin{lem}\label{lemgminfcon}The subset $\Gamma_{\infty}$ of $X\times X$ is constructible.
	\end{lem}

\begin{proof} Let us prove, by an induction on $\dim(Y)$, that $\pi_1^{-1}(Y)\cap \Gamma_{\infty}$ is constructible for every irreducible 
subvariety $Y\subseteq X$. 
By convention, set $\dim Y=-1$ when $Y=\emptyset$. So, the case $\dim Y=-1$ is trivial. Now assume that $\dim Y\geq 0$
and that the result holds in dimension $<\dim(Y)$. Set $Z_Y=\overline{\langle W\rangle\cdot \Delta_Y}$; this set is invariant 
under the action of $\langle W\rangle$ on $X\times X$.
Since $\overline{W^i\cdot\Delta_Y}$ is irreducible and increasing for each $i\geq 0$, there is $m\geq 0$, such that 
\begin{equation}
Z_Y=\overline{\langle W\rangle\cdot \Delta_Y}=\overline{W^i\cdot \Delta_Y}\quad  (\forall  i\geq m).
\end{equation} 
Then there is a dense open subset $U_Y$ of $Z_Y$ which is contained in $W^m\cdot \Delta_Y$, hence in $\langle W\rangle\cdot \Delta_Y$. Shrinking $U_Y$ if necessary, we may assume that  $\pi_1(U_Y)$ is open in $Y$. 
Then $Y\setminus \pi_1(U_Y)$ is a closed subset of $X$, the irreducible components of which have dimension $< \dim Y$. By the induction hypothesis, $\pi_1^{-1}(Y\setminus \pi_1(U_Y))\cap \Gamma_{\infty}$ is constructible. We also know that $\pi_1^{-1}(\pi_1(U_Y))\cap \Gamma_{\infty}=\langle W \rangle\cdot U_Y$ is an open subset of $Z_Y$. Thus, $\pi_1^{-1}(Y)\cap \Gamma_{\infty}=(\pi_1^{-1}(Y\setminus \pi_1(U_Y))\cap \Gamma_{\infty})\cup (\pi_1^{-1}(\pi_1(U_Y))\cap \Gamma_{\infty})$  is constructible.
\end{proof}

\begin{pro}\label{pro:dim-orbits} The orbits $W^k(x)$ satisfy the following properties. 
\begin{enumerate}
\item The function $k\in \Z_{>0}\mapsto \dim(W^k(x))$ is non-decreasing.
\item The function $x\in X\mapsto \dim(W^k(x))$ is  lower semi-continuous in the Zariski topology: the subsets
	$\{x\in X\, ; \; \dim(W^k(x))\leq n\}$ are Zariski closed for all pairs $(n,k)$ of integers. 
\item The integers 
$$
s(x):= \max_{k\geq 0} \{  \dim(W^k(x))  \}\quad \text{and} \quad s_X:= \max_{x\in X} \{ s(x)\}$$ are bounded from above by $\dim(X)$.
\item There is a Zariski dense open subset $\sU$ of $X$ and an integer $k_0$ such that
	$\dim(W^k(x)) = s_X$ for all $k\geq k_0$ and all $x\in \sU$.
\item There is an integer $\ell\geq 0$, such that for every $x$ in $X$, $W^\ell(x)=\langle W \rangle (x)$ and $W^\ell(x)$ 
is an open subset of $\overline{\langle W \rangle (x)}$.
\end{enumerate}
\end{pro}

This result and its proof below are analogous to~\cite[Prop. 7.1.2]{FK18} and~\cite[Sec. 1]{AFKKZ}. 

\begin{proof} 
The first assertion follows from the inclusions~\eqref{eq:inclusion-Wk}, and the third one is obvious.
Since  the action $(f,x)\in W^k\times X \mapsto f(x)\in X$ is a morphism, 
the second and fourth assertions follow from Chevalley's constructibility result 
and  the semi-continuity of the dimension of the fibers (see \cite[II, Exercise 3.19]{Hartshorne1977} 
and \cite[Section I.6.3, Corollary]{Shafarevich1} respectively).  	 
By Lemma \ref{lemgminfcon}, $\Gamma_{\infty}$ is constructible. Since it is the increasing union of the  constructible subsets $\Gamma_i$, there is an integer $\ell$ such that $\Gamma_\infty=\Gamma_i$ for $i\geq \ell$. 
Then, $W^\ell(x)=\langle W \rangle (x)$ because  $W^i(x)=\pi_2(\Gamma_i\cap \pi_1^{-1}\{x\})$ and
$\langle W\rangle (x)=\pi_2(\Gamma_\infty\cap \pi_1^{-1}\{x\})$.
Now, the constructible set  $W^\ell(x)$ contains a dense open subset  $U$ of $\overline{\langle W \rangle (x)}$; since $\langle W\rangle$ acts transitively on $W^\ell(x)$, $W^\ell(x)=\langle W\rangle (U)$ is open in $\overline{\langle W \rangle (x)}$.		 
\end{proof}

\subsection{Open orbits}\label{par:s=d}
Let us assume in this paragraph that $s_X=\dim X$: there is an orbit
$W^{k}(x_0)$ which is open and dense and coincides with $\langle W \rangle (x_0)$. We fix such a pair $(k,x_0)$. 
Let $f$ be an element of $\langle W \rangle$. Since  $f(x_0)$ is in the set $W^{k}(x_0)$, there is an element 
$g$ of $W^k$ such that $g(x_0)=f(x_0)$,  i.e. $g^{-1}\circ f(x_0)=x_0$. 
By commutativity, $(g^{-1}\circ f)(h(x_0))=h(x_0)$ for every $h$ in $W^{k}$, and this shows that $g^{-1}\circ f=\id$ 
because $W^k(x_0)$ is dense in $X$. Thus, $\langle W\rangle$ coincides with $W^k$, and 
$\langle W\rangle=\langle V\rangle$ is an irreducible algebraic subgroup of the ind-group $\Aut(X)$. 

Thus, Theorem~B is proved in case $s_X=\dim X$. The proof when $s_X< \dim X$ occupies the next section, and is achieved in 
\S~\ref{par:conclusion-thmB}.

%	Let us study the structure of this group. Since $V$ is a commutative family, $\langle V\rangle$ is an abelian group. 
%	Since $X$ is an affine variety, and $\langle V\rangle$ is an algebraic group acting faithfully on $X$,
%	$\langle V\rangle$ is an affine algebraic group (see~\cite{Popov-Vinberg}, Section 1.2). The classification of abelian, connected, linear algebraic groups tells us that  $\langle V\rangle$ is isomorphic to $(\G_a^r)\times (\G_m^s)$
%	for some pair of integers $r$, $s \geq 0$ with $r+s=\dim \langle V\rangle$. 
%	
%	\begin{pro}
%		If $V$ is a commutative, irreducible, algebraic subset of $\Aut(X)$, and if $s_X=\dim(X)$, then 
%		$\langle V\rangle$ is a linear algebraic group isomorphic to $(\G_a^r)\times (\G_m^{s})$ for some integers $r, s\geq 0$ with $r+s=s_X$, acting faithfully on $X$.
%	\end{pro}

%%%%%%%%%%%%%%%%%%% 
\subsection{No dense orbit}
%%%%%%%%%%%%%%%%%%%
Assume now that there is no dense orbit; in other words, $s_X< \dim(X)$. 
Fix an integer $\ell>0$ and a $W$-invariant open subset $\sU\subset X$ 
such that  
\begin{equation}
s(x)=s_X \; {\text{ and }} W^\ell(x)=\langle W\rangle(x)
\end{equation}
for every $x\in \sU$ (see Proposition~\ref{pro:dim-orbits}, assertions~(4) and~(5)).

%Fix an integer $\ell>0$ and a $W$-invariant open subset $\sU\subset X$ 
%such that  
%\begin{equation}
%s(x)=s_X \; {\text{ and }} W^\ell(x)=\langle W\rangle(x)
%\end{equation}
%for every $x\in \sU$ (see Proposition~\ref{pro:dim-orbits}, assertions~(4) and~(5)).
	
\subsubsection{A fibration} We start with 
a construction which is reminiscent of Rosenlicht's quotient theorem~\cite{Rosenlicht:1956}: instead of looking at 
orbits of an algebraic group $G$, we consider ``orbits'' of the commutative set of transformations $W^\ell$. 

Let $C$ be an irreducible algebraic subvariety of $X$ of codimension $s_X$ that intersects the general orbit $W^\ell(x)$ transversally (in $k$ points).  As in \S~\ref{par:orbits}, denote by $\pi_1: X\times X \to X$   the projection to the first factor. 
The morphism 
\begin{equation}\pi':=(\pi_1)|_{(X\times C)\cap \Gamma_{\ell}}: (X\times C)\cap \Gamma_{\ell}\to X
\end{equation} is generically finite of degree $k$.
So there is a non-empty open subset $\sV$ of $\sU$ such that $\pi'|_{\pi'^{-1}(\sV)}: \pi'^{-1}(\sV)\to \sV$ is finite \'etale.  
Observe that for every $g\in \langle W\rangle$, $g(\sV)$ is open in $\sU$ and $\pi'|_{\pi'^{-1}(g(\sV))}$ is finite \'etale of degree $k$. Set $Y:=\langle W\rangle(\sV)$; it is open in $\sU$ and satisfies
\begin{itemize}	
\item[(i)]  for each $x \in Y$
the intersection of $C$ and $W^\ell(x)$ is transverse and contains exactly $k$
points;	
\item[(ii)] $Y$ is $W$-invariant.
\end{itemize}

To each point $x\in Y$, we associate the intersection $C\cap W^\ell(x)$, viewed as a point in the space $C^{[k]}$ of cycles of length $k$ and dimension $0$ in 
$C$. This gives a dominant morphism 
\begin{equation}
\pi \colon Y\to B
\end{equation}
where, by definition, $B$ is the irreducible variety $B= \overline{\pi(Y)}\subset C^{[k]}$.
The group $\langle W\rangle$ is now contained 
in $\Aut_\pi(Y)$. Shrinking $B$ and $Y$ accordingly, we may assume that $B$ is normal
and that $\pi$
is surjective.  Let $\eta$ be the generic point of $B.$
	
The fiber $\pi^{-1}(b)$ of $b\in B$, we denote by $Y_b$. By construction, for every $b\in B(\bfk)$,
$Y_b$ is an orbit of $\langle W\rangle$; 
and Section~\ref{par:s=d} shows that $Y_b$ is isomorphic to the image $\langle W\rangle_b$
 of $\langle W\rangle$  in $\Aut(Y_b)$: this group $\langle W\rangle_b$ coincides with
the image of $W^\ell$  in $\Aut(Y_b)$ and the action of $\langle W\rangle$ on $Y_b$ corresponds to the action of $\langle W\rangle_b$ on itself 
by translation. 
Thus, Section~\ref{par:s=d} implies the following properties 
\begin{enumerate}
\item every fiber of $\pi$, in particular its generic fiber, is geometrically irreducible;  
\item the generic fiber of $\pi$ is normal and affine, shrinking $B$ (and Y accordingly) again, we may assume $B$ and $Y$ 
to be normal and affine;
\item  the action of $\langle W\rangle$ on the generic fiber $Y_\eta$ has bounded degree. 
\end{enumerate}

\subsubsection{Reduction to $Y= U_B\times_B (\G_{m,B}^s)$}\label{par:Reduction_of_Y} 
In this section, the variety $Y$ will be modified, so as to reduce our study to the case when $Y$ is an abelian group scheme over~$B$. Note that $B$ and $Y$ will be modified several times in this paragraph, keeping the same names.

By Proposition \ref{proweil}, after shrinking $B$, there exist an embedding $\tau: Y\hookrightarrow \bbA_{B}^N$ for some $N\geq 0$ and a homomorphism $\rho:  \langle W\rangle\hookrightarrow \GL_{N}(\sO(B))    \subseteq \Aut_B(\bbA_{B}^N)$ such that  
\begin{equation}
\tau \circ g=\rho(g)\circ \tau \; \quad (\forall g\in \langle W\rangle).
\end{equation}
Via $\tau$, we view $Y$ as a $B$-subscheme of $\bbA_{B}^N$, and via $\rho$ we view $\langle W\rangle$ in $\GL_{N}(\sO(B))$.
Consider the inclusion of $\GL_N(\sO(B))$ into $\GL_N(\bfk(B))$, and compose it with the embedding of $W$ into $\GL_N(\sO(B))$. Denote by $\langle W\rangle_{\eta}$ the Zariski closure of $\langle W\rangle$ in $\GL_{N}(\bfk(B),Y_{\eta})\subseteq \Aut(Y_{\eta})$, where $\GL_{N}(\bfk(B),Y_{\eta})$ is the subgroup of $\GL_{N}(\bfk(B))$ which preserves $Y_{\eta}$. 
There is a natural inclusion of sets $W\hookrightarrow W\otimes_{\bfk}\bfk(B)$: a point $x$ of $W$, viewed as a morphism $x\colon \Spec \bfk(x) \to W$, is mapped to the point
\begin{equation}
x^B: \Spec \bfk(x)(B\otimes_{\bfk}\bfk(x)) =\Spec {\rm Frac}(\bfk(x)\otimes_{\bfk} \bfk(B)) \to W\otimes_{\bfk} \bfk(B),
\end{equation}
where $\bfk(x)(B\otimes_{\bfk}\bfk(x))$ is the function field of $B\otimes_{\bfk}\bfk(x)$ which is the variety over the field $\bfk(x)$; note that $\bfk$ being algebraically closed, 
$B\otimes_\bfk\bfk(x)$ is irreducible over $\bfk(x)$ and $\bfk(x)\otimes_{\bfk} \bfk(B)$ is an integral domain.
The image of this inclusion is Zariski dense in $W\otimes_{\bfk}\bfk(B)$.
The morphism $W\hookrightarrow \GL_{N}(\bfk(B),Y_{\eta})$ naturally extends to a morphism $W\otimes_{\bfk}\bfk(B)\hookrightarrow \GL_{N}(\bfk(B),Y_{\eta})$.
It follows that $\langle W\rangle_{\eta}$ is the Zariski closure of $\langle W\otimes_{\bfk}\bfk(B)\rangle$ in $\GL_{N}(\bfk(B),Y_{\eta})$. 

Since $W\otimes_{\bfk}\bfk(B)$ is geometrically irreducible, 
$\langle W\rangle_{\eta}$ is a geometrically irreducible commutative linear algebraic group over $\bfk(B)$. 
As a consequence (\cite{Milne}, Chap. 16.b), there exists a finite extension $L$ of $\bfk(B)$ and an integer $s\geq 0$ such that 
\begin{equation}\label{equl}
\langle W\rangle_{\eta}\otimes_{\bfk(B)}L\simeq U_{L}\times \G_{m, L}^s 
\end{equation}
where $U_L$ is a unipotent commutative linear algebraic group over $L$.  

Let $\psi:B'\to B$ be the normalization of $B$ in $L$.
We obtain a new fibration $\pi'\colon Y\times_B B' \to B'$, together with an embedding $\iota_\psi$ of $\Aut_\pi(Y)$ in $\Aut_{\pi'}(Y\times_B B')$; by Lemma~\ref{lem:basechange}, the subgroup 
$\langle W\rangle$ has bounded degree if and only if its image $\iota_\psi\langle W\rangle$ has bounded degree too.
Because the generic fiber of $\pi$ is geometrically irreducible, $Y\times_B B'$ is irreducible.
After such a base change, we may assume that $\langle W\rangle_{\eta}\simeq U_{\eta}\times \G_{m, \bfk(B)}^s,$
 where $U_{\eta}$ corresponds to the group $U_L$ of Equation~\eqref{equl}.
Replacing (this new) $B$ by an affine open subset, and shrinking $Y$ accordingly, we may assume that $Y= U_B\times_B (\G_{m,B}^s)$, where $U_B$ is 
an integral unipotent commutative algebraic group scheme over $B$, and 
\begin{equation}\label{eqinclude}
W\subseteq U_B(B)\times \G_{m,B}^s(B)\subseteq \Aut_\pi(Y)
\end{equation} 
 acts on $Y$ by translation;   here $U_B(B)$ and $\G_{m,B}^s(B)$ denote the ind-varieties of sections of the structure  morphisms 
 $U_B \to B$ and $\G_{m,B}^s \to B$ respectively.  
 
 \begin{rem}  
A section $\sigma \colon B \to U_B$ defines an automorphism of $U_B \simeq B\times_B U_B$  
by $\phi(\sigma\times_B\id_{U_B})$, where $\phi \colon U_B \times U_B \to U_B$ is the multiplication morphism of $U_B$; it defines in the same way an element of $\Aut_\pi(Y)$. Similarly $\G_{m,B}^s(B)$ embeds into $\Aut_\pi(Y)$, so $U_B(B)\times \G_{m,B}^s(B)\subseteq \Aut_\pi(Y)$, 
and this is the meaning of \eqref{eqinclude}. 
\end{rem}

\begin{rem} 
Both $U_B(B)\times \G_{m,B}^s(B)$ and $\Aut_\pi(Y)$ are ind-varieties over $\bfk$ and the inclusions 
in \eqref{eqinclude} are morphisms between ind-varieties. 
\end{rem}

Now, to prove Theorem~B, we only need to show that $W$ is contained in an 
algebraic subgroup of $U_B(B)\times \G_{m,B}^s(B)$.  

\subsubsection{Structure of $U_B$} 
Let $B$ be a normal affine variety over the algebraically closed field $\bfk$, and let $U_B$ be an integral, connected and unipotent algebraic group scheme over $B$ (we do not assume $U_B$ to be commutative here).

%{\color{red}We call an ind-group \textbf{nested} if it is an increasing
%union of algebraic subgroups.}

\begin{lem}\label{lemunialg}The ind-group $U_B(B)$ is an increasing
union of algebraic subgroups.
\end{lem}

In the language of \cite{FK18}, Lemma~\ref{lemunialg} says that $U_B(B)$  is a nested ind-group (see~Remark~\ref{rem:intro-nested}).
Before describing the proof, let us assume that $U_B$ is just an $r$-dimensional additive group $\G_{a, B}^r$. 
Then, each element of $U_B$ can be written 
\begin{equation}
f = (a_1^f(z),\dots,a_r^f(z))
\end{equation}
where each $a_i^f(z)$ is an element of $\sO(B)$; 
its $n$-th power 
is given by $f^n=(n a_1^f(z),\dots,\\ na_r^f(z))$. 
Thus, viewed as automorphisms of $Y$, the degrees of the $f^n$ are bounded independently of $n$, by 
(a function of) the degrees of the $a_i^f$. 
Our proof is a variation on this basic remark, with two extra difficulties: the structure of $U_B$ 
may be more subtle in positive characteristic (see~\cite{Serre:GACDC}, \S VII.2); instead of iterating one element $f$, we need to controle 
the group $U_B$ itself.

\begin{proof}
Denote by $\pi_U:U_B\to B$ the structure morphism. Recall, from the end of Section~\ref{par:Reduction_of_Y}, that $B$ is an affine
variety.

The proof is by induction on the relative dimension of $\pi_U\colon U_B\to B$. If this dimension is zero, 
there is nothing to prove. So, we assume that the lemma holds for  relative dimensions $\leq \ell$, for some $\ell \geq 0$, 
and we want to prove it when the relative dimension is $\ell+1$.
Denote by $U_{\eta}$ the generic fiber of  $\pi_U$. 
Our field $\bfk$ is algebraically closed, and the group $U_B$ is 
connected, so by Corollary~14.55 of \cite{Milne} (see also \S~14.63),  there exists a finite field extension $L$ of
$\bfk(B)$ such that  $U_{L}:=U_{\eta}\otimes_{\bfk(B)}L$
sits in a central exact sequence
\begin{equation}
0\to \G_{a,L}\to U_L
\xrightarrow{q_L} V_L\to 0,
\end{equation}
where $V_L$ is an irreducible unipotent group of dimension $\ell$ and 
$V_L$ is isomorphic to $\bbA^\ell_L$ as an $L$-variety; moreover,  there is an isomorphism 
of  $L$-varieties $\phi_L\colon U_L\to V_L\times \G_{a,L}$ such that the quotient 
morphism $q_L$ is given by the projection onto the first factor. So we have a section $s_L: V_L\to U_L$ such that 
$q_L\circ s_L=\id$. The section $s_L$ is just given by a regular function on $V_L$, it needs not be a homomorphism of groups.
Doing the base change given by the normalization of $B$ in $L$, and then shrinking the base if necessary, 
we may assume that $B$ is affine and
\begin{itemize}
\item there is an exact sequence of group schemes over  $B$,
$$0\to \G_{a,B}\rightarrow U_B\xrightarrow[]{q_B} V_B\to 0,$$
where $V_B$ is a unipotent group scheme over $B$ of relative dimension $\ell$; 
\item there is an isomorphism of $B$-schemes $V_B\simeq \bbA^\ell_B$;
\item $s_L$ extends 
to a section $s_B: V_B\to U_B$ over $B$: $q_B\circ s_B=\id$.
\end{itemize}
For $b\in B$, denote by $U_b$, $V_b$, $q_b$, $s_b$ the specialization of~$U_B$, $V_B$, $q_B$, $s_B$ at~$b$.
%Since $U_B$ and $V_B$ are abelian, 
We denote by $\circ_U$ and $\circ_V$ the group laws on the groups $U_B$ and $V_B$ respectively.  
The morphism of $B$-schemes $\beta: U_B\to V_B\times \G_{a,B}$ 
sending a point $x$ in the fiber $U_b$ to the point $(q_b(x), x-s_b(q_b(x)))$ of the fiber $V_b\times \G_{a,b}$ defines an isomorphism. 
We use $\beta$ to transport the group law of $U_B$ into $V_B\times \G_{a,B}$; this defines a law $*$ on $V_B\times \G_{a,B}$, given by 
\begin{equation}
a_1*a_2=\beta(\beta^{-1}(a_1) \circ_U \beta^{-1}(a_2)),
\end{equation}
for $a_1$ and $a_2$ in $V_B\times \G_{a,B}$.   
Denote by $O(V_B\times_B V_B)$ the function ring of the $\bfk$-variety $V_B\times_B V_B\simeq B \times \bbA^{\ell}\times \bbA^{\ell}$.
We write a point in $V_B\times_B V_B$ as $(b,x_1,x_2)$ where $x_1,x_2\in V_B$ with the same image $b$ in $B$.
There is an element $F(b,x_1,x_2)(y_1,y_2)$ of $O(V_B\times_B V_B)[y_1,y_2]$ such that
\begin{equation}\label{eq:law_on_VG}
(x_1,y_1)*(x_2,y_2)=(x_1 \circ_V x_2,F(b,x_1,x_2)(y_1,y_2))
\end{equation}
for all $b\in B$ and $(x_1,y_1),(x_2,y_2)\in V_b\times \bbG_{a}$. %(In \eqref{eq:law_on_VG}, $+$ is the group law in $V_b$.) 
For every fixed $(x_1,y_1,x_2)$,  the morphism $y_2\mapsto F(b,x_1,x_2)(y_1,y_2)$  is an automorphism of the variety $\G_{a}$. Thus, we can write 
\begin{equation}
F(b,x_1,x_2)(y_1,y_2)=C_0(b,x_1,x_2)(y_1)+C_2(b,x_1,x_2)(y_1)y_2.
\end{equation}
The function $C_2(b,x_1,x_2)(y_1)$ does not vanish on $V_B\times_B V_B\times \bbA^1\simeq B\times \bbA^{2\ell+1}$;
thus, $C_2$ is an element of $\sO(B)$.
By  symmetry we get
\begin{equation}
F(b,x_1,x_2)(y_1,y_2)=C_0(b,x_1,x_2)+C_1(b)y_1 +C_2(b)y_2
\end{equation}
and
\begin{equation}
(x_1,y_1)*(x_2,y_2)=(x_1 \circ_V x_2,C_0(b,x_1,x_2)+C_1(b)y_1 +C_2(b)y_2).
\end{equation}
Now, apply this equation for $x_1=x_2=0$ (the neutral element of $V_B$). The restriction of $\beta$ 
to the fiber $q_b^{-1}(0)$ is $x\mapsto x-s_b(0)$, so for $(0,y_1)$ and $(0, y_2)$ in $V_b\times \G_a$, we obtain $(0,y_1)*(0, y_2)=(0, y_1+y_2+s_b(0))$; then  $C_1=C_2=1$. 

We  identify now the ind-varieties $U_B(B)$  and $V_B(B)\times \bbG_a(B)$. By induction, the ind-group $V_B(B)$ is an
increasing union of algebraic subgroups $V_i$; as observed before the proof of this lemma, the ind-group $\bbG_a(B)$ is an increasing union of subgroups $G_j$. If $S$ and $T$ are elements of $V_B(B)$ and $\bbG_a(B)$ respectively, we set
\begin{equation}
\delta_V(S)=\min \{i\; ; \; S\in V_i\}, \quad \delta_{\bbG_a}(T)=\min \{j\; ; \; T\in G_j\}.
\end{equation}
Each element of $U_B(B)$
is given by a section $(S,T)\in V_B(B)\times \bbG_a(B)$ and the group law in $U_B(B)$ corresponds to the law 
\begin{equation}\label{eq:3.16}
(S_1,T_1)*(S_2,T_2)=(S_1 \circ_V S_2,C_0(S_1,S_2)+T_1 +T_2)
\end{equation}
because $C_1=C_2=1$. Here $C_0: V_B(B)\times V_B(B)\to \G_a(B)$ is a morphism of ind-varieties, so there is a
function $\alpha\colon \N \to \N$ such that  
\begin{equation}
\delta_{\bbG_a} C_0(S_1,S_2)\leq \alpha(\delta_V(S_1)+\delta_V(S_2)).
\end{equation}
Now, note that $V_i\times G_{\alpha(2i)}$ is an algebraic subgroup of $U_B(B)$, because 
\begin{eqnarray}
\quad\quad\quad  \delta_{\bbG_a}(C_0(S_1,S_2)+T_1 +T_2) & \leq &  \max\{\delta_{\bbG_a}(C_0(S_1,S_2)),\delta_{\bbG_a}(T_1),\delta_{\bbG_a}(T_2)\}.
\end{eqnarray}
Thus, $U_B(B)$ is the increasing union of the algebraic subgroups $V_i\times G_{\alpha(2i)}$.\end{proof}

\subsubsection{Subgroups of $\G_m^s(B)$ and conclusion}\label{par:conclusion-thmB}

\begin{lem}\label{lemgm}
If $Z$ is an irreducible subvariety of $\G_m^s(B)$ containing $\id$, then 
 $\langle Z\rangle$ is an algebraic subgroup of $\G_m^s(B)$.
\end{lem}

This lemma may be derived from~\cite[Proposition 4.4.1.]{FK18}; it means that $(\G_m^s(B))^\circ$ is nested.  We provide the  proof for completeness.  	

\begin{proof}[Proof of Lemma \ref{lemgm}]
Pick a projective compactification ${\overline{B}}$ of $B$. After taking the normalization of ${\overline{B}}$, 
we may assume  ${\overline{B}}$ to be normal. 
If $h$ is any non-constant rational function on $\overline{B}$,  denote by ${\rm Div}(h)$ the divisor $(h)_0-(h)_\infty$  on ${\overline{B}}$. 

Let  $\bfy=(y_1, \ldots, y_{s})$ be the standard coordinates on $\G_m^{s}$. 
Each element $f\in \G_m^s(B)$ can be written as $(b_1^f(z),\dots,b_{s}^f(z))$, for some
$b_j^f\in \sO^*(B)$. 
Let $R$ be an effective divisor
whose support $\supp(R)$ 
contains ${\overline{B}}\setminus B$. Replacing $R$ by some large multiple,  $Z$ is contained 
in the subset $P_R$ of $\G_m^s(B)$ made of automorphisms $f \in \G_m^s(B)$ such that
${\rm Div}(b_i^f)+R\geq 0$ and ${\rm Div}(1/b_i^f)+R\geq 0$ for all $i=1,\dots,s$.
Let us study the structure of this set $P_R\subset \G_m^s(B)$.

Let $K$ be the set  of pairs $(D_1,D_2)$ of effective divisors supported on ${\overline{B}}\setminus B$ such that $D_1$ and $D_2$ have no common irreducible component, $D_1\leq R$, $D_2\leq R$, and $D_1$ and $D_2$ are rationally equivalent.
Then $K$ is a finite set. For every pair $\alpha=(D^{\alpha}_1,D^{\alpha}_2)\in K$, we choose a function $h_{\alpha}\in \sO^*(Y)$ such that ${\rm Div}(h_{\alpha})=D^{\alpha}_1-D^{\alpha}_2$;  if $h$ is another element of  $\sO^*(Y)$ such that ${\rm Div}(h)=D^{\alpha}_1-D^{\alpha}_2$, then $h/h_{\alpha}\in \bfk^*$.
By convention $\alpha=0$ means that $\alpha=(0,0)$, and in that case we choose $h_\alpha$ to be the constant function $1$. 	
For every $\beta=(\alpha_1,\dots,\alpha_{s})\in K^{s}$, denote by $P_{\beta}$ the set of elements $f\in \G_m^s(B)$ such that
the $b_i^f\in \sO^*(B)$ satisfy ${\rm Div}(b_i^f)=D^{\alpha_i}_1-D^{\alpha_i}_2$ for all $i=1,\ldots, s$.
Then $P_{\beta}\simeq \G_m^{s}(\bfk)$ is an irreducible algebraic variety over $\bfk$. Moreover, $\id\in P_{\beta}$ if and only if $\beta=0$, and $P_{0}$ is an algebraic subgroup of $\G_m^s(B)$, isomorphic to $\G_m^{s}(\bfk)$ as an algebraic group. 
	
Observe that $P_R$ is the disjoint union 
$P_R=\bigsqcup_{\beta\in K^{s}}P_{\beta}$. Since $\id\in Z$, $Z$ is irreducible, and $Z\subseteq P_R$, we obtain $Z\subset  P_0$.  
Since $P_0$ is an algebraic subgroup of $\G_m^s(B)$, $\langle Z\rangle$ coincides with
$(Z\cdot Z^{-1})^\ell$ for some $\ell \geq 1$, and $\langle Z\rangle$ is a connected algebraic group.
\end{proof}
	
\begin{proof}[Proof of Theorem~B]	
By Proposition~\ref{ind-group-bded-degree}, we only need to prove that $W=\langle V\rangle$ is of bounded degree. 
By Lemma \ref{lem:basechange}
 $W$ is a subgroup of bounded degree if and only if $W \subset \Aut_{\pi}(Y)$ is a subgroup of bounded degree. Moreover,
 by \eqref{eqinclude}, $W$ is a subgroup of $U_B(B)\times \G_{m}^s(B) \subset \Aut_{\pi}(Y)$.
	Denote by $\pi_1: U_B(B)\times \G_{m}^s(B)\to U_B(B)$ the projection to the first factor and 
	$\pi_2: U_B(B)\times \G_{m}^s(B)\to \G_{m}^s(B)$ the projection to the second. By Lemma \ref{lemunialg}, there exists an algebraic  subgroup $H_1$ of $U_B(B)$ containing $\pi_1(W)$.
	Since $\pi_2(W)$ is irreducible and contains $\id$,  Lemma \ref{lemgm} shows that  $\pi_2(W)$ is contained in  an algebraic  subgroup $H_2$ of $\G_m(B)$. Then $W$ is contained in the algebraic subgroup $H_1\times H_2$ of $U_B(B)\times \G_{m}^s(B)$. This concludes the proof.
\end{proof}	
	
%%%%%%%%%%%%%%%%%%%%%%%%%%%%%%%	
\section{Actions of additive groups}
%%%%%%%%%%%%%%%%%%%%%%%%%%%%%%%
	
\begin{thm}\label{corconnuncountablecommutator} Let $\bfk$ be an uncountable, algebraically closed field. 
Let $X$ be a connected affine variety over $\bfk$. 
Let $G\subset \Aut(X)$ be an algebraic subgroup isomorphic to $\G_a^r$, for some $r\geq 1$. 	
Let $H=\{h\in \Aut(X)|\,\, gh=hg \text{ for every } g\in G\}$ be the centralizer of $G$. 
If $H/G$ is at most countable then $G$ acts simply transitively on $X$, so that $X$ is isomorphic to $G$ as a $G$-variety.
\end{thm}
	
This section is devoted to the proof of this result. A proof is described in \cite[\S 10.4]{FK18}
when $X$ is irreducible and the characteristic of $\bfk$ is $0$; we just explain how to extend the proof of Furter and Kraft. 

\begin{proof}
Let $X_1$ be an irreducible component of $X$ on which $G$ acts non-trivially. Denote by $X_j$, $j\geq 2$ the remaining components.

Suppose that $G$ acts transitively on $X_1$. Then $X=X_1$, because otherwise $X_1$ would intersect another component of $X$ on a proper $G$-invariant set; so, the statement is proved in that case. 
We  now assume   towards a contradiction        that $G$ does not act transitively on $X_1$. Pick a $G$-orbit $O_1\subset X_1$, and set $Z_1=\overline{O_1}$ or $Z_1=\overline{X_1\setminus O_1}$ if $\overline{O_1}=X_1$. By construction, $Z_1$ is a proper, closed, and $G$-invariant subset of $X_1$. Hence, the ideal $I_1\subset O(X)$ of functions vanishing on $Z_1$ and  on each of the $X_i$ for $i\neq 1$ is not reduced to $0$.
If we choose a function $f_1$ in $I_1\setminus\{0\}$, its $G$-orbit generates a $G$-invariant, finite dimensional 
subspace of $I_1$ (see~\cite[\S 1.2]{Popov-Vinberg}); since $G$ is isomorphic to $\G_a^r$, there is a non-zero invariant 
vector $f$ in this space (this is an instance of the Lie-Kolchin theorem). Such a function is 
not constant since it vanishes on $Z_1$. This implies that $I_1^G$ is an
infinite dimensional vector space over $\bfk$, for it contains $f\bfk[f]$.

Identify $G(\bfk)$ to the vector space $\bfk^r$, and pick an element $g_0\in G(\bfk)$
that acts non-trivially on $X_1$. To each $s$ in $I_1^G$ we associate
the map  $x\in X(\bfk) \mapsto s(x)g_0\in  \bfk^r$ and the automorphism of $X$ defined by 
$F_s(x) = (s(x)g_0)(x)$. 
If $F_s=\id_X$, then $g_0$ is an element of the stabilizer $G_x$ for every $x$ at which $s(x)\neq 0$. 
If $s$ vanishes on a proper subset of $X_1$, this implies that $g_0$ acts trivially on $X_1$, a contradiction. 
So, $F_s\neq \id_X$ for $s\neq 0$. This means that 
$F\colon s\in I_1^G\mapsto F_s\in H$ is injective. 

Now, by Proposition 10.4.4(1) of \cite{FK18}, the homomorphism $F$ is a morphism of 
ind-groups. Thus, $H$ contains an infinite dimensional ind-group. Since $\bfk$ is not
countable, we get a contradiction and we are done.
\end{proof}

%
%%%%%%%%%%%%%%%%%%%%%%%%%%%%%%%%%%%%%%%%%%%%%%%%%%%%%%%%%%%%%%%%%%
%		
\section{Proof of Theorem~A}\label{CharacterizationofAn}
%
%%%%%%%%%%%%%%%%%%%%%%%%%%%%%%%%%%%%%%%%%%%%%%%%%%%%%%%%%%%%%%%%%%
%		
	
In this section, we prove Theorem~A. So, $\bfk$ is an uncountable, algebraically closed field, 
$X$ is a connected affine algebraic variety over~$\bfk$, 
and $\varphi: \Aut(\bbA^n_\bfk) \to \Aut(X)$ is an isomorphism of (abstract) groups.  
	
%%%%%%%%%%%%%%%%%%%%%%%%%%%%%%%
\subsection{Translations and dilatations}	
%%%%%%%%%%%%%%%%%%%%%%%%%%%%%%%
Let   $\Tr \subset \Aut(\bbA^n_\bfk)$ be the group of all translations and 
$\Tr_i$ the subgroup of translations of the $i$-th coordinate:
\begin{equation}
(x_1, \ldots, x_n) \mapsto (x_1,\dots,x_i + c,\dots,x_n)
\end{equation}
for some $c$ in $\bfk$.  Let
$\T\subset \GL_n(\bfk)\subset \Aut(\bbA^n_\bfk)$ be the diagonal group (viewed as a maximal torus) and let
$\T_i$ be the subgroup of automorphisms  
\begin{equation}
(x_1, \ldots, x_n)\mapsto (x_1,\dots,ax_i,\dots,x_n)
\end{equation}
for some $a\in \bfk^*$. A direct computation shows that {\sl{$\Tr$ (resp. $\T$) coincides with its centralizer in 
$\Aut(\bbA^n_\bfk)$}}.

\begin{lem}\label{lemcountablesub} Let $G$ be a subgroup of  $\Tr$ whose index is at most countable. 
Then, the centralizer of $G$ in $\Aut(\bbA^n)$ is $\Tr$.
\end{lem}
	
\begin{proof} 
The centralizer of $G$ contains $\Tr$. Let us prove the reverse inclusion.
The index of $G$ in $\Tr$ being at most countable, $G$ is Zariski dense in $\Tr$.
Thus, if $h$ centralizes $G$, we get $hg=gh$ for all $g\in \Tr$, and $h$ is in fact in the centralizer of $\Tr$.
 Since $\Tr$ coincides with its centralizer, we get $h\in\Tr$.  
\end{proof}
	
%%%%%%%%%%%%%%%%%%%%%%%%%%%%%%%
\subsection{Closed subgroups}	
%%%%%%%%%%%%%%%%%%%%%%%%%%%%%%%
As in Section~\ref{ind-groups}, we endow $\Aut(X)$ with the structure of an ind-group, given by a filtration by algebraic varieties $\Aut_j$ for $j \ge 1$. 
	
\begin{lem}\label{closed}
The  groups $\varphi(\Tr)$, $\varphi(\Tr_i)$, $\varphi(\T)$ and $\varphi(\T_i)$ are closed subgroups of
$\Aut(X)$ for all $i = 1,\dots,n$.
\end{lem}
\begin{proof}
Since $\Tr \subset \Aut(\bbA^n_\bfk)$ coincides with its centralizer, 
$\varphi(\Tr) \subset \Aut(X)$ coincides with its centralizer too and, as such, is a closed subgroup of $\Aut(X)$. 
The same argument applies to $\varphi(\T) \subset \Aut(X)$.
To prove that $\varphi(\Tr_i) \subset \Aut(X)$ is closed we note that $\varphi(\Tr_i)$ is the subset of elements $f \in \varphi(\Tr)$ 
		that commute  to every element $g\in \varphi(\T_j)$ for every index $
		j\neq i$  in $\{1, \ldots, n\}$. 
		Analogously,  $\varphi(\T_i) \subset \Aut(X)$ is a closed subgroup because an element $f$ of $\T$ is in $\T_i$ if and only if it commutes to all elements $g$ of $\Tr_j$ for $j\neq i$.\end{proof}
	
	\newpage
%%%%%%%%%%%%%%%%%%%%%%%%%%%%%%%
\subsection{Proof of Theorem~A}	
%%%%%%%%%%%%%%%%%%%%%%%%%%%%%%%
	
\subsubsection{Abelian groups (see~\cite{Milne, Serre:GACDC})}\label{par:uni-p}
Before starting the proof, let us recall a few important facts on abelian, affine algebraic groups. Let $G$ be an algebraic group over 
the field $\bfk$, such that $G$ is abelian, affine, and connected.
\begin{enumerate}
\item If $\char (\bfk)=0$, then $G$ is isomorphic to $\G_a^r\times \G_m^s$ for some pair of integers $(r,s)$; if $G$ is 
unipotent, then $s=0$. (see~\cite{Serre:GACDC}, \S VII.2, p.172)
\item  If $\char (\bfk)>0$, then $G$ is a product of a multiplicative type subgroup $G_s$ and a unipotent subgroup $G_u$ (see \cite{Milne}, Theorem 17.17).
Moreover, since  $\bfk$ is algebraically closed,  $G_s$ is isomorphic to an algebraic torus $\G_m^s$ for some $s \ge 0$.
\end{enumerate}
We list two 
criteria on the $p$-torsion elements of a commutative connected algebraic group $G$ that may rigidify the structure of $G_s$ and $G_u$: 
\begin{enumerate}
\item[(3)] If $\char (\bfk)=p>0$, $G$ is unipotent, and all
 non-trivial 
elements of $G$ have order $p$, then $G$ is isomorphic to $\G_a^r$ 
for some $r\geq 0$.  (see~\cite{Serre:GACDC}, \S VII.2, Prop. 11, p.178)
\item[(4)]  If $\char (\bfk)=p>0$, and there is no non-trivial element in $G$ of order $p^\ell$, for any $\ell \geq 0$, then $G$ is isomorphic to $G_s=\G_m^s$ 
for some  $s\geq 0$. (see~\cite{Milne}, Theorem 16.13 and Corollary 16.15, and~\cite{Serre:GACDC}, \S VII.2, p.176)
\end{enumerate}
To keep examples in mind, note that all non-trivial elements of $\Tr_1(\bfk)$ have order $p$ and $\T_1(\bfk)$ does not contain any non-trivial 
element of order $p^\ell$  when $\char(\bfk)=p$.
	
\subsubsection{Proof of Theorem~A} Let us now prove Theorem~A. 

By Lemma \ref{closed}, $\varphi(\Tr_1) \subset \Aut(X)$ is a closed subgroup; in particular, $\varphi(\Tr_1)$ is an ind-subgroup of $\Aut(X)$. Let $\varphi(\Tr_1)^\circ$ be the connected component of the identity of  $\varphi(\Tr_1)$; from Section~\ref{par:ind_groups_connected}, we know that the index of $\varphi(\Tr_1)^\circ$ in  $\varphi(\Tr_1)$ is at most countable. The ind-group $\varphi(\Tr_1)^\circ$  is an increasing union $\cup_i V_i$ of irreducible algebraic varieties $V_i$,  each $V_i$ containing the identity. 
Theorem~B implies that each $\langle V_i \rangle$ is an irreducible algebraic subgroup of $\Aut(X)$.  
Since  $\varphi(\Tr_1)$ does not contain non-trivial elements of order $k<\infty$ with 
$k\wedge \char(\bfk)=1$, it follows from  properties (1) and (2) of Section~\ref{par:uni-p} that $\langle V_i \rangle$ is unipotent;
moreover, by properties (1) and (3) of  Section~\ref{par:uni-p}, $\langle V_i \rangle$ is isomorphic to 
$\G_a^{r_i}$ for some~$r_i$. 
Thus 
\begin{equation}
\varphi(\Tr_1)^\circ=\cup_{i\geq 0}F_i
\end{equation} 
where the $F_i$ form an increasing family of unipotent algebraic subgroups of $\Aut(X)$, each of them isomorphic to some $\G_a^{r_i}$. 
We may assume that $\dim F_0\geq 1$.

Similarly, $\varphi(\T_1)^\circ \subset \varphi(\T_1)$ is a subgroup of countable index  and 
\begin{equation}
\varphi(\T_1)^\circ = \cup_{i\geq} G_i,
\end{equation}
where  the
$G_i$ are increasing irreducible commutative algebraic subgroups of $\Aut(X)$ (we do not assert that $G_i$ is of type $\G_m^{s_i}$ yet). 
We may assume that $\dim G_0\geq 1.$

The group $\T_i$ acts by conjugation on $\Tr_i$ for every $i\leq n$, this action has exactly two orbits $\{ 0 \}$ and $\Tr_i\setminus \{ 0 \}$, and the
action on $\Tr_i\setminus \{ 0 \}$ is free; hence, the same properties
hold for the action of  $\varphi(\T_i)$ on $\varphi(\Tr_i)$ by conjugation. 

Let $H_i$ be the subgroup of $\varphi(\Tr_1)$ generated by all $g\circ f\circ g^{-1}$ with $f$ in $F_i$ and $g$ in $G_i$. 
Theorem~B shows that $H_i$ is an irreducible algebraic subgroup of $\varphi(\Tr_1)$. We have $H_i\subseteq H_{i+1}$ and $g\circ H_i\circ g^{-1}=H_i$
for every $g\in G_i$. 	
	
Write $H_i=\G_a^l$ for some $l\geq 1$. We claim that $G_i\simeq \G_a^r\times \G_m^s$ for a pair of integers $r$, $s\geq 0$ with $r+s\geq 1$. 
This follows from properties~(1), ~(2) and~(3)  of Section~\ref{par:uni-p} because, when $\char(\bfk)=p>1$, the only element in $\varphi(\T_1)$ of order $p^\ell$, $\ell\geq 0$, is the identity element.
Since the action of $\varphi(\T_1)$ on $\varphi(\Tr_1 \setminus \{ 0 \})$ is free, the action of $G_i$ on $F_i\setminus \{0\}$ is free too,
 and thus, we get an action of $\G_a^r$ by automorphisms of the algebraic group $\G_a^l$ without fixed point in $\G_a^l\setminus\{0\}$, and this forces $r=0$ (an instance of the Lie-Kolchin theorem).  
Let $q$ be a prime number with $q\wedge \char(\bfk)=1$. Then $\G_m^s$ contains a copy of $(\Z/q\Z)^s$, and $\T_1$ does not contain 
such a subgroup if $s>1$; so, $s=1$, $G_i\simeq \G_m$ and $G_i=G_{i+1}$ for all $i\geq 0$. It follows that $\varphi(\T_1)^{\circ}\simeq \G_m$.
Since the index of $\varphi(\T_1)^\circ$ in $\varphi(\T_1)$ is  countable, there exists a countable subset $I\subseteq \varphi(\T_1)$
such that $\varphi(\T_1)=\sqcup_{h\in I} \varphi(\T_1)^{\circ}\circ h$.

Let $f\in F_i$ be a nontrivial element.  
Since the action of $\varphi(\T_1)$ on $\varphi(\Tr_1 \setminus \{ 0 \})$ is transitive,  
\begin{equation}
F_i\setminus \{0\}=\bigcup_{h\in I}   \left( \left(\bigcup_{g\in \varphi(\T_1)^{\circ}} (g\circ h)\circ f \circ (g\circ h)^{-1} \right)     \cap F_i \right) .
\end{equation}
The right hand side is a countable union of subvarieties of $F_i\setminus \{0\}$ of dimension at most one. It follows that $\dim F_i=1$, $F_i\simeq \G_a$,  and $\varphi(\Tr_1)^{\circ}\simeq \G_a$. Thus, we have
\begin{equation}
\varphi(\Tr_1)^{\circ}\simeq \G_a, \; {\text{and}}\; \varphi(\T_1)^{\circ}\simeq \G_m.
\end{equation}

Since each $\varphi(\Tr_i)^{\circ}$ is isomorphic to $\G_a$, $\varphi(\Tr)^{\circ}$ is an $n$-dimensional commutative unipotent group and its index in  $\varphi(\Tr)$ is   at most countable. 
By Lemma~\ref{lemcountablesub}, the centralizer of $\varphi^{-1}(\varphi(\Tr)^{\circ})$ in $\Aut(\bbA^n_\bfk)$ is $\Tr$. It follows that the centralizer of $\varphi(\Tr)^{\circ}$ in $\Aut(X)$ is $\varphi(\Tr)$.
Then Theorem~\ref{corconnuncountablecommutator} implies that $X$ is isomorphic to $\bbA^n_\bfk$.

%
%%%%%%%%%%%%%%%%%%%%%%%%%%%%%%%%%%%%%%%%%%%%
%
\section{Appendix: the degree functions for rational self-maps}\label{par:appendix}
%
%%%%%%%%%%%%%%%%%%%%%%%%%%%%%%%%%%%%%%%%%%%%
%

Here, we follow~~\cite{Dang-degrees-2020, Truong} to prove a general version of Lemma~\ref{lem:basechange}. As above, 
$\bfk$ is an algebraically closed field. We first start with the case of projective varieties.

\subsection{Degree functions on projective varieties}
 
Let $X$ be a projective and normal variety over $\bfk$ of pure dimension $d=\dim(X)$. 
Let $H$ be a big and nef divisor on $X$. For every dominant 
rational self-map $f$ of $X$,
and every $j=0,\dots,d$,    set
\begin{equation}
\deg_{j,H}f =(f^*(H^j)\cdot H^{d-j}).
\end{equation}
Pick a normal resolution of $f$; by this we mean  a projective and normal variety $\Gamma$,
a birational morphism $\pi_1:\Gamma\to X$ and a morphism $\pi_2: \Gamma\to X$  satisfying $f=\pi_2\circ\pi_1^{-1}.$
Then we have $\deg_{j,H}f=(\pi_2^*(H^j)\cdot \pi_1^*(H^{d-j}))> 0$, for $f$ is dominant.
Let $L$ be another big and nef divisor. There is $c>1$ such that $cL-H$  and $cH-L$ are big. 
Then we have 
$ 
\deg_{j,H}f=(\pi_2^*(H^j)\cdot \pi_1^*(H^{d-j}))\leq c^d(\pi_2^*(L^j)\cdot \pi_1^*(L^{d-j}))=c^d\deg_{j,L}f.
$  
Symetrically, we get $\deg_{j,L}f\leq (c')^d\deg_{j,H}f$ for some $c'>1$. Thus, two big and nef divisors
give rise to comparable degree functions: 
\begin{equation}
C^{-1} \deg_{j,H}(f)\leq \deg_{j,L} (f)\leq C \deg_{j,H} (f) \quad \quad (\forall 0\leq j \leq d)
\end{equation}
for all rational dominant maps $f\colon X\dasharrow X$, and some $C>1$.

\begin{lem}\label{lemprogenfi}
Let $Y$ be a projective and normal variety over $\bfk$ of pure dimension~$d$. 
Let $\pi: Y\dashrightarrow X$ be a dominant and generically finite rational map.
Let $H$ and $L$ be   big and nef divisors, on $X$ and $Y$ respectively.
Then there is a constant $C>1$ such that 
for every $j=0,\dots, d$, and every pair of dominant rational self-maps  $f: X\dashrightarrow X$ and $g: Y\dashrightarrow Y$ satisfying $f\circ \pi=\pi\circ g$, we have 
$$C^{-1}\deg_{j,L}(g)\leq \deg_{j, H}(f)\leq C\deg_{j, L}(g).$$
\end{lem}
\proof 
Denote by $x_1,\dots, x_s$ the generic points of $X$ and $y_1,\dots, y_r$ the generic points of $Y$. Since $\pi$ is 
dominant and generically finite, there is a surjective map $\sigma:\{1,\dots,r\}\to\{1,\dots, s\}$
such that $\pi(y_i)=x_{\sigma(i)}, i=1,\dots,r$. For every $i=1,\dots, r$, set $t_i=\deg[\bfk(y_i): \pi^*\bfk(x_{\sigma(i)})]$
and then 
\begin{equation}
m=\min_{i=1\dots, s}(\sum_{l\in \sigma^{-1}(i)}t_l), \quad m'=\max_{i=1\dots, s}(\sum_{l\in \sigma^{-1}(i)}t_l)
\end{equation}
Take a resolution of $\pi$, defined by a projective and normal variety $Z$,
a birational morphism $\pi_1:Z\to Y$ and a morphism $\pi_2: Z\to X$  satisfying $\pi=\pi_2\circ\pi_1^{-1}.$
Set $h:=\pi_1^{-1}\circ g \circ\pi_1: Z\dashrightarrow Z$. For each index $0\leq j \leq d$, the projection formula (see \cite{Dang-degrees-2020}, Theorem 2.3.2(iv) and the references therein, notably
 \cite{Fulton1998}, Proposition 1.7)  gives
\begin{equation}\label{equationghbir}
\deg_{j,L}g=\deg_{j,\pi_1^*L}h  
\end{equation}
\begin{equation}\label{equationfhdon}
m \deg_{j, H}f \leq \deg_{j,\pi_2^*H}h \leq m' \deg_{j, H}f.
\end{equation}
Since $\pi_1^*L$ and $\pi_2^*H$ are big and nef on $Z$, there is a constant $C_1>1$  that depends only on $\pi_1^*L$ and $\pi_2^*H$ such that 
\begin{equation}\label{equationhhl}C_1^{-1}\deg_{j,\pi_2^*H}h\leq \deg_{j,\pi_1^*L}h\leq C_1\deg_{j,\pi_2^*H}h.
\end{equation}
We conclude the proof by combining the last three equations.
\endproof

\subsection{Equivalent functions}\label{par:equivalent}  
Let $S$ be a set. We shall say that two functions $F,G: S\to \R_{\geq 0}$, are {\bf{equivalent}} if there is 
a constant $C>1$ such that 
\begin{equation}
C^{-1}\max\{G,1\}\leq \max\{F,1\}\leq C\max\{G,1\},
\end{equation}
where $\max\{G,1\}$ denotes the maximum between $G$ and $1$.
We denote by $[F]$ the equivalence class of $F$; the equivalence class $[1]$ coincides with the set of bounded functions $S\to \R_{\geq 0}$.

\subsection{Degree functions on varieties} Now, let $X$ be a variety of pure dimension $d$ over $\bfk$. 
Let $\pi\colon Z\dasharrow X$ be a  
birational map
such that $Z$ is projective and normal, and let $H$ be a big and nef divisor on $Z$. Then, define
the degrees $\deg_{j,H}f$ of any rational dominant map $f\colon X\dasharrow X$ by $\deg_{j,H}f=\deg_{j,H}\pi^{-1}\circ f\circ \pi$. The previous paragraph shows that if we change the model $(Z,\pi)$ or the divisor $H$ (to $H'$), 
then we get two notions of 
degrees $\deg_{j,H}$ and $\deg_{j,H'}$ which are equivalent functions, in the sense of \S~\ref{par:equivalent}, on the set
of rational dominant self-maps of $X$. This justifies the following definition.

Let $S$ be a family of dominant rational maps $f_s\colon X\dasharrow X$, $s\in S$. 
A {\bf{notion of degree}} on $S$ in codimension $j$ is a function $\deg_j\colon S\to \R_{\geq 0}$ 
in the equivalence class $[\deg_{j,H}]$ for some normal projective model $Z\to X$  
and some big and nef divisor $H$ on $Z$. The equivalence class $[\deg_j]$ is unique.

\begin{rem}
Assume further that $X$ is affine.
In Section \ref{par:degrees}, we defined a notion of degree  $f\mapsto \deg f$  (in codimension $1$) on the set of automorphisms of $X$; this notion depends on an embeding $X\hookrightarrow \bbA^N_{\bfk}, N\geq 0.$
However, its equivalence class on $\Aut(X)$ does not depend on the choice of such an embedding and is equal to 
the class $[\deg_1]$ defined in this section. \end{rem}

From Lemma~\ref{lemprogenfi} 
and the definitions, we obtain:

\begin{pro}\label{progenfinidegrefun}
Let $\pi\colon Y\dasharrow X$ be a dominant and generically finite rational map between two varieties $X$ and $Y$ over $\bfk$,
each of pure dimension $d$. Let $S$ be a family of dominant rational maps  $g_s\colon Y\dasharrow Y$ such that 
for every $s$ in $S$ there is a rational map $f_s\colon X\dasharrow X$ that satisfies $\pi\circ g_s = f_s\circ \pi$. 
Then, for each $j=0, \ldots, d$, the equivalence classes of the degree functions  $s\in S \mapsto \deg_j(g_s)$ and $s\in S\mapsto \deg_j(f_s)$ are equal.
\end{pro}

\bibliographystyle{plain}	
\bibliography{dd}
	
\end{document}